\documentclass[runningheads]{llncs}
\usepackage[utf8]{inputenc}

\usepackage{amsmath}
\usepackage{amssymb}
\usepackage[dvipsnames]{xcolor}
\usepackage{setspace}
\usepackage{enumerate}
\usepackage{enumitem}
\usepackage{stmaryrd}

\newcommand{\tr}{\mathrm{tr}}
\newcommand{\pow}{\mathrm{pow}}
\newcommand{\inv}{\mathrm{inv}}
\newcommand{\diag}{\mathrm{diag}}

\newcommand{\Sym}{\mathrm{Sym}}
\newcommand{\Frob}{\mathrm{Frob}}
\newcommand{\TO}{\rightarrow}
\newcommand{\lto}{\longrightarrow}
\newcommand{\lmto}{\longmapsto}
\newcommand{\ls}{\leqslant}
\newcommand{\dotprod}[2]{\langle #1|#2\rangle}

\newcommand{\E}{\mathbb{E}}
\newcommand{\R}{\mathbb{R}}
\newcommand{\N}{\mathbb{N}}
\newcommand{\M}{\mathcal{M}}
\newcommand{\mc}{\mathcal}
\newcommand{\mf}{\mathfrak}
\newcommand{\gs}{\geqslant}

\begin{document}

\title{Exploration of Balanced Metrics on\\ Symmetric Positive Definite Matrices}
\titlerunning{Exploration of Balanced Metrics on SPD matrices}
\author{Yann Thanwerdas \and Xavier Pennec}
\authorrunning{Y. Thanwerdas, X. Pennec}

\institute{Universit\'e C\^ote d'Azur, Inria, Epione, France}

\maketitle

\begin{abstract}
    Symmetric Positive Definite (SPD) matrices have been used in many fields of medical data analysis. Many Riemannian metrics have been defined on this manifold but the choice of the Riemannian structure lacks a set of principles that could lead one to choose properly the metric. This drives us to introduce the principle of balanced metrics that relate the affine-invariant metric with the Euclidean and inverse-Euclidean metric, or the Bogoliubov-Kubo-Mori metric with the Euclidean and log-Euclidean metrics. We introduce two new families of balanced metrics, the mixed-power-Euclidean and the mixed-power-affine metrics and we discuss the relation between this new principle of balanced metrics and the concept of dual connections in information geometry.
\end{abstract}

\section{Introduction}

Symmetric Positive Definite (SPD) matrices are used in many applications: for example, they represent covariance matrices in signal or image processing \cite{Barachant13,Deligianni11,Cheng13} and they are diffusion tensors in diffusion tensor imaging \cite{lenglet_statistics_2006,pennec_riemannian_2006,Fletcher07}. Many Riemannian structures have been introduced on the manifold of SPD matrices depending on the problem and showing significantly different results from one another on statistical procedures such as the computation of barycenters or the principal component analysis. Non exhaustively, we can cite Euclidean metrics, power-Euclidean metrics \cite{Dryden10}, log-Euclidean metrics \cite{Arsigny06}, which are flat; affine-invariant metrics \cite{pennec_riemannian_2006,Fletcher07,Lenglet06} which are negatively curved; the Bogoliubov-Kubo-Mori metric \cite{Michor00} whose curvature has a quite complex expression.

Are there some relations between them? This question has practical interests. First, understanding the links between these metrics could lead to interesting formulas and allow to perform more efficient algorithms. Second, finding families of metrics that comprise these isolated metrics could allow to perform optimization on the parameters of these families to find a better adapted metric. Some relations already exist. For example, the power-Euclidean metrics \cite{Dryden10} (resp. power-affine metrics \cite{Thanwerdas19}) comprise the Euclidean metric (resp. affine-invariant metric) and tend to the log-Euclidean metric when the power tends to zero.

We propose the principle of \textit{balanced metrics} after observing two facts. The affine-invariant metric $g^A_\Sigma(X,Y) = \tr((\Sigma^{-1}X\Sigma^{-1})Y)$ on SPD matrices 
may be seen as a balanced hybridization of 
the Euclidean metric $g^E_\Sigma(X,Y)=\tr(XY)$ on one vector and of the inverse-Euclidean metric (the Euclidean metric on precision matrices) $g^I_\Sigma(X,Y)=\tr((\Sigma^{-1}X\Sigma^{-1})(\Sigma^{-1}Y\Sigma^{-1}))$ on the other vector. Moreover, the definition of the Bogoliubov-Kubo-Mori metric can be rewritten as  $g^{BKM}_\Sigma(X,Y)=\tr(\partial_X\log(\Sigma)\,Y)$ where it appears as a balance of the Euclidean metric and the log-Euclidean metric $g^{LE}_\Sigma(X,Y)=\tr(\partial_X\log(\Sigma)\,\partial_Y\log(\Sigma))$. These observations raise a few questions. Given two metrics, is it possible to define a balanced bilinear form in general? If yes, is it clear that this bilinear form is symmetric and positive definite? If it is a metric, are the Levi-Civita connections of the two initial metrics dual in the sense of information geometry?

In this work, we explore this principle through the affine-invariant metric, the Bogoliubov-Kubo-Mori metric and we define two new families of balanced metrics, the mixed-power-Euclidean and the mixed-power-affine metrics. In section 2, we show that if a balanced metric comes from two flat metrics, the three of them define a dually flat structure. In particular, we show that the balanced structure defined by the Euclidean and the inverse-Euclidean metrics corresponds to the dually flat structure given by the $\pm1$-connections of Fisher information geometry. In section 3, we enlighten the balanced structure of the BKM metric and we generalize it by defining the family of mixed-power-Euclidean metrics. In section 4, we define the family of mixed-power-affine metrics and we discuss the relation between the concepts of balanced metric and dual connections when the two initial metrics are not flat.

\section{Affine-invariant metric as a balance of Euclidean and inverse-Euclidean metrics}

Because the vocabulary may vary from one community to another, we shall first introduce properly the main geometric tools that we use in the article (Section 2.1). Then we examine in Section 2.2 the principle of balanced metric in the particular case of the pair Euclidean / inverse-Euclidean metrics and we formalize it in the general case of two flat metrics. In Section 2.3, we show that the $\pm 1$-connections of the centered multivariate normal model are exactly the Levi-Civita connections of the Euclidean and inverse-Euclidean metrics.

\subsection{Reminder on metrics, connections and parallel transport}

On a manifold $\M$, we denote $\mc{C}^\infty(\M)$ the ring of smooth real functions and $\mf{X}(\M)$ the $\mc{C}^\infty(\M)$-module of vector fields.

\subsubsection{Connection}
A \textit{connection} is an $\R$-bilinear map $\nabla:\mf{X}(\M)\times\mf{X}(\M)\lto\mf{X}(\M)$ that is $\mc{C}^\infty(\M)$-linear in the first variable and satisfies the Leibniz rule in the second variable. It gives notions of parallelism, parallel transport and geodesics. A vector field $V$ is \textit{parallel} to the curve $\gamma$ if $\nabla_{\dot{\gamma}}V=0$. The \textit{parallel transport} of a vector $v$ along a curve $\gamma$ is the unique vector field $V_{\gamma(t)}=\Pi_\gamma^{0\TO t}v$ that extends $v$ and that is parallel to $\gamma$. Thus, the connection is an infinitesimal parallel transport, that is $\Pi_\gamma^{t\TO 0}V_{\gamma(t)}=V_{\gamma(0)}+t\nabla_{\dot{\gamma}}V+o(t)$. The \textit{geodesics} are autoparallel curves, that is curves $\gamma$ satisfying $\nabla_{\dot{\gamma}}\dot{\gamma}=0$.

\subsubsection{Levi-Civita connection}
Given a metric $g$ on a manifold $\M$, the \textit{Levi-Civita connection} is the unique torsion-free connection $\nabla^g$ compatible with the metric $g$, that is $\nabla^gg=0$ or more explicitly $X(g(Y,Z))=g(\nabla^g_XY,Z)+g(Y,\nabla^g_XZ)$ for all vector fields $X,Y,Z\in\mf{X}(\M)$. Thus a metric inherits notions of parallel transport and geodesics. Note that geodesics coincide with distance-minimizing curves with constant speed.

\subsubsection{Dual connections}
Given a metric $g$ and a connection $\nabla$, the \textit{dual connection} of $\nabla$ with respect to $g$ is the unique connection $\nabla^*$ satisfying the following equality $X(g(Y,Z))=g(\nabla_XY,Z)+g(Y,\nabla^*_XZ)$ for all vector fields $X,Y,Z\in\mf{X}(\M)$. It is characterized by Lemma \ref{dualconnections} below. In this sense, the Levi-Civita connection $\nabla^{g}$ is the unique torsion-free self-dual connection with respect to $g$. We say that $(\M,g,\nabla,\nabla^*)$ is a \textit{dually-flat manifold} when $\nabla,\nabla^*$ are dual with respect to $g$ and $\nabla$ is flat (then $\nabla^*$ is automatically flat \cite{Amari00}).

\begin{lemma}[Characterization of dual connections]\label{dualconnections}
Two connections $\nabla,\nabla'$ with parallel transports $\Pi,\Pi'$ are dual with respect to a metric $g$ if and only if the dual parallel transport preserves the metric, i.e. for all vector fields $X,Y\in\mf{X}(\M)$ and all curve $\gamma$, $g_{\gamma(t)}(X_{\gamma(t)},Y_{\gamma(t)})=g_{\gamma(0)}(\Pi_\gamma^{t\TO 0}X_{\gamma(t)},(\Pi')_\gamma^{t\TO 0}Y_{\gamma(t)})$.
\end{lemma}

\begin{proof}
The direct sense is proved in \cite{Amari00}. Let us assume that the dual parallel transport preserves the metric and let $X,Y,Z\in\mf{X}(\M)$ be vector fields. Let $x\in\M$ and let $\gamma$ be a curve such that $\gamma(0)=x$ and $\dot{\gamma}(0)=X_x$. Using the first order approximation of the parallel transport, our assumption leads to:
\begin{align*}
    g_{\gamma(t)}(Y_{\gamma(t)},Z_{\gamma(t)}) &=g_x(\Pi_\gamma^{t\TO 0}Y_{\gamma(t)},(\Pi')_\gamma^{t\TO 0}Z_{\gamma(t)})\\
    &=g_x(Y_x+t\nabla_{\dot{\gamma}}Y+o(t),Z_x+t\nabla'_{\dot{\gamma}}Z+o(t))\\
    &=g_x(Y_x,Z_x)+t[g_x(\nabla_{\dot{\gamma}}Y,Z)+g_x(Y,\nabla'_{\dot{\gamma}}Z)]+o(t).
\end{align*}
So $X_x(g(Y,Z))=g_x(\nabla_{\dot{\gamma}}Y,Z)+g_x(Y,\nabla'_{\dot{\gamma}}Z)$ and $\nabla,\nabla'$ are dual w.r.t. $g$. $\Box$
\end{proof}

\subsection{Principle of balanced metrics}

\subsubsection{Observation}
We denote $\M={SPD}_n$ the manifold of SPD matrices and $N=\dim\M=\frac{n(n+1)}{2}$. The (A)ffine-invariant metric $g^A$ on SPD matrices \cite{pennec_riemannian_2006,Fletcher07,Lenglet06}, i.e. satisfying $g^A_{M\Sigma M^\top}(MXM^\top,MYM^\top)=g^A_\Sigma(X,Y)$ for $M\in{GL}_n$, is defined by:
\begin{equation}\label{A}
    g^A_\Sigma(X,Y) = \tr((\Sigma^{-1}X\Sigma^{-1})Y)=\tr(X(\Sigma^{-1}Y\Sigma^{-1})).
\end{equation}
The (E)uclidean metric $g^E$ on SPD matrices is the pullback metric by the embedding $id:\M\hookrightarrow(\Sym_n,\dotprod{\cdot}{\cdot}_\Frob)$:
\begin{equation}\label{E}
    g^E_\Sigma(X,Y) = \tr(XY).
\end{equation}
The (I)nverse-Euclidean metric $g^I$ on SPD matrices belongs to the family of power-Euclidean metrics \cite{Dryden10} with power $-1$. If SPD matrices are seen as covariance matrices $\Sigma$, the inverse-Euclidean metric is the Euclidean metric on precision matrices $\Sigma^{-1}$:
\begin{equation}\label{I}
    g^I_\Sigma(X,Y) = \tr(\Sigma^{-2}X\Sigma^{-2}Y) = \tr((\Sigma^{-1}X\Sigma^{-1})(\Sigma^{-1}Y\Sigma^{-1})).
\end{equation}
Observing these definitions, the affine-invariant metric (\ref{A}) appears as a \textit{balance} of the Euclidean metric (\ref{E}) and the inverse-Euclidean metric (\ref{I}). We formalize this idea thanks to parallel transport.

\subsubsection{Formalization}

The diffeomorphism $\inv:(\M,g^I)\lto(\M,g^E)$ is an isometry. Since these two metrics are flat, the parallel transports do not depend on the curve. On the one hand, the Euclidean parallel transport from $\Sigma$ to $I_n$ is the identity map $\Pi^E:X\in T_\Sigma\M\lmto X\in T_{I_n}\M$ since all tangent spaces are identified to the vector space of symmetric matrices $\Sym_n$ by the differential of the embedding $id:\M\hookrightarrow\Sym_n$. On the other hand, the isometry $\inv$ gives the inverse-Euclidean parallel transport from $\Sigma$ to $I_n$, $\Pi^{I}:X\in T_\Sigma\M\lmto\Sigma^{-1}X\Sigma^{-1}\in T_{I_n}\M$. We generalize this situation in Definition \ref{balanced}. Given Lemma \ref{dualconnections}, it automatically leads to Theorem \ref{dual}.

\begin{definition}[Balanced bilinear form]\label{balanced}
Let $g,g^*$ be two flat metrics on ${SPD}_n$ and $\Pi,\Pi^*$ the associated parallel transports that do not depend on the curve. We define the balanced bilinear form $g^0_\Sigma(X,Y)=\tr((\Pi_{\Sigma\TO I_n}X)(\Pi^*_{\Sigma\TO I_n}Y))$.
\end{definition}

\begin{theorem}[A balanced metric defines a dually flat manifold]\label{dual}
Let $g,g^*$ be two flat metrics and let $\nabla,\nabla^*$ be their Levi-Civita connections. If the balanced bilinear form $g^0$ of $g,g^*$ is a metric, then $(\M,g^0,\nabla,\nabla^*)$ is a dually flat manifold.
\end{theorem}

If two connections $\nabla$ and $\nabla^*$ are dual connections with respect to a metric $g^0$, there is no reason for them to be Levi-Civita connections of some metrics. Therefore, the main advantage of the principle of balanced metrics on the concept of dual connections seems to be the metric nature of the dual connections.

\begin{corollary}[Euclidean and inverse-Euclidean are dual with respect to affine-invariant]
We denote $\nabla^E$ and $\nabla^I$ the Levi-Civita connections of the Euclidean metric $g^E$ and the inverse-Euclidean metric $g^I$. Then $g^A$ is the balanced metric of $g^I,g^E$ and $({SPD}_n,g^A,\nabla^I,\nabla^E)$ is a dually flat manifold.
\end{corollary}

\subsection{Relation with Fisher information geometry}

We know from \cite{Skovgaard84} that the affine-invariant metric is the Fisher metric of the centered multivariate normal model. Information geometry provides a natural one-parameter family of dual connections, called $\alpha$-connections \cite{Amari00}. In the following table, we recall the main quantities characterizing the \textit{centered multivariate normal model} $\mc{P}=\{p_\Sigma:\R^n\lto\R_+^*,\Sigma\in\M\}$, where $\M={SPD}_n$.
$$\begin{array}{cc}
    \mathrm{Densities} & ~~p_\Sigma(x)=\frac{1}{{\sqrt{2\pi}}^n}\frac{1}{\sqrt{\det\Sigma}}\exp\left(\frac{1}{2}x^\top\Sigma^{-1}x\right) \\[1.5mm]
    \mathrm{Log~likelihood} & ~~l_\Sigma(x)=\log p_\Sigma(x)=\frac{1}{2}\left(-n\log(2\pi)-\log\det\Sigma+x^\top\Sigma^{-1}x\right)\\[1.5mm]
    \mathrm{Differential} & ~~d_\Sigma l(V)(x)=-\frac{1}{2}\left[\tr(\Sigma^{-1}V)+x^\top\Sigma^{-1}V\Sigma^{-1}x\right]\\[1.5mm]
    \mathrm{Fisher~metric} & ~~g_\Sigma(V,W)=\frac{1}{2}\tr(\Sigma^{-1}V\Sigma^{-1}W)
\end{array}$$
We recall that the \textit{$\alpha$-connections} $\nabla^{(\alpha)}$ \cite{Amari00} of a family of densities $\mc{P}$ are defined by their Christoffel symbols $\Gamma_{ijk}=g_{lk}\Gamma_{ij}^l$ in the local basis $(\partial_i)_{1\ls i\ls N}$ at $\Sigma\in\M$:
\begin{equation}
    \left(\Gamma_{ijk}^{(\alpha)}\right)_\Sigma=\E_\Sigma\left[\left(\partial_i\partial_j l+\frac{1-\alpha}{2}\partial_i l\partial_j l\right)\partial_k l\right].
\end{equation}

We give in Theorem \ref{alpha} the expression of the $\alpha$-connections of the centered multivariate normal model and we notice that the Euclidean and inverse-Euclidean Levi-Civita connections belong to this family.

\begin{theorem}[$\alpha$-connections of the centered multivariate normal model]\label{alpha}
In the global basis of $\M={SPD}_n$ given by the inclusion $\M\hookrightarrow\Sym(n)\simeq\R^N$, writing $\partial_XY=X^i(\partial_i Y^j)\partial_j$, the $\alpha$-connections of the multivariate centered normal model are given by the following formula:
\begin{equation}\label{alphaconnection}
    \nabla^{(\alpha)}_XY=\partial_XY-\frac{1+\alpha}{2}(X\Sigma^{-1}Y+Y\Sigma^{-1}X).
\end{equation}
The mixture $m$-connection ($\alpha=-1$) is the Levi-Civita connection of the Euclidean metric $g^E_\Sigma(X,Y)=\tr(XY)$. The exponential $e$-connection ($\alpha=1$) is the Levi-Civita connection of the inverse-Euclidean metric, i.e. the pullback of the Euclidean metric by matrix inversion, $g^I_\Sigma(X,Y)=\tr(\Sigma^{-2}X\Sigma^{-2}Y)$.
\end{theorem}

The formula (\ref{alphaconnection}) can be proved thanks to Lemma \ref{expectations} which gives the results of expressions of type $\int_{\R^n}{x^\top\Sigma^{-1}X\Sigma^{-1}Y\Sigma^{-1}Z\Sigma^{-1}x\exp\left(-\frac{1}{2}x^\top\Sigma^{-1}x\right)dx}$, with $y=\Sigma^{-1/2}x$, $A=\Sigma^{-1/2}X\Sigma^{-1/2}$, $B=\Sigma^{-1/2}Y\Sigma^{-1/2}$ and $C=\Sigma^{-1/2}Z\Sigma^{-1/2}$. If one wants to avoid using the third formula of Lemma \ref{expectations}, one can rely on the formula (\ref{alphaconnection}) in the case $\alpha=0$ which is already known from \cite{Skovgaard84}.

\begin{lemma}\label{expectations}
For $A,B,C\in\Sym_n$:
\begin{align*}
    \E_{I_n}(y\lmto y^\top Ay) &=\tr(A),\\
    \E_{I_n}(y\lmto y^\top Ayy^\top By) &=\tr(A)\tr(B)+2\tr(AB),\\
    \E_{I_n}(y\lmto y^\top Ayy^\top Byy^\top Cy) &=\tr(A)\tr(B)\tr(C)+8\tr(ABC)\\
    & \quad +2(\tr(AB)\tr(C)+\tr(BC)\tr(A)+\tr(CA)\tr(B)).
\end{align*}
\end{lemma}

\begin{proof}[Theorem \ref{alpha}]
Given Lemma \ref{expectations}, the computation of the Christoffel symbols $\left(\Gamma^{(\alpha)}_{ijk}\right)_\Sigma$ leads to $\left(\Gamma^{(\alpha)}_{ijk}\right)_\Sigma X^iY^jZ^k=-\frac{1+\alpha}{4}\tr(\Sigma^{-1}[X\Sigma^{-1}Y+Y\Sigma^{-1}X]\Sigma^{-1}Z)$.
On the other hand, the relation $\Gamma_{ijk}=g_{lk}\Gamma_{ij}^l$ between Christoffel symbols gives $\left(\Gamma^{(\alpha)}_{ijk}\right)_\Sigma X^iY^jZ^k=\frac{1}{2}\tr\left(\Sigma^{-1}\left[\left(\Gamma_{ij}^l\right)_\Sigma^{(\alpha)} X^iY^j\partial_l\right]\Sigma^{-1}Z\right)$.
So we get:
\begin{equation}
    \nabla^{(\alpha)}_XY=\partial_XY+\left[\left(\Gamma_{ij}^l\right)_\Sigma^{(\alpha)} X^iY^j\partial_l\right]=\partial_XY-\frac{1+\alpha}{2}(X\Sigma^{-1}Y+Y\Sigma^{-1}X).
\end{equation}

It is clear that the mixture connection ($\alpha=-1$) is the Euclidean connection. The inverse-Euclidean connection can be computed thanks to the Koszul formula. This calculus drives exactly to the exponential connection ($\alpha=1$). $\Box$
\end{proof}

In the next section, we apply the principle of balanced metrics to the pairs Euclidean / log-Euclidean (Bogoliubov-Kubo-Mori metric) and power-Euclidean / power-Euclidean (mixed-power-Euclidean metrics).

\section{The family of mixed-power-Euclidean metrics}

\subsection{Bogoliubov-Kubo-Mori metric}

The Bogoliubov-Kubo-Mori metric $g^{BKM}$ is a metric on symmetric positive definite matrices used in quantum physics. It was introduced as $g_\Sigma^{BKM}(X,Y)=\int_0^\infty{\tr((\Sigma+tI_n)^{-1}X(\Sigma+tI_n)^{-1}Y)dt}$ and can be rewritten \cite{Michor00} thanks to the differential of the symmetric matrix logarithm $\log:\M={SPD}_n\lto\Sym_n$ as:
\begin{equation}\label{BKM}
    g_\Sigma^{BKM}(X,Y)=\tr(\partial_X\log(\Sigma)\,Y)=\tr(X\,\partial_Y\log(\Sigma)).
\end{equation}
The log-Euclidean metric $g^{LE}$ \cite{Arsigny06} is the pullback metric of the Euclidean metric by the symmetric matrix logarithm $\log:(\M,g^{LE})\lto(\Sym_n,g^E)$:
\begin{equation}\label{LE}
    g^{LE}_\Sigma(X,Y)=\tr(\partial_X\log(\Sigma)\,\partial_Y\log(\Sigma)).
\end{equation}
Therefore, the BKM metric (\ref{BKM}) appears as the balanced metric of the Euclidean metric (\ref{E}) and the log-Euclidean metric (\ref{LE}). As the Euclidean and log-Euclidean metrics are flat, the parallel transport does not depend on the curve and Theorem \ref{dual} ensures that they form a dually flat manifold.

\begin{corollary}[Euclidean and log-Euclidean are dual with respect to BKM]\label{dualBKM}
We denote $\nabla^{E}$ and $\nabla^{LE}$ the Levi-Civita connections of the Euclidean metric $g^E$ and the log-Euclidean metric $g^{LE}$. Then $g^{BKM}$ is the balanced metric of $g^{LE},g^E$ and $({SPD}_n,g^{BKM},\nabla^{LE},\nabla^E)$ is a dually flat manifold.
\end{corollary}

\subsection{Mixed-power-Euclidean}

Up to now, we observed that existing metrics (affine-invariant and BKM) were the balanced metrics of pairs of flat metrics (Euclidean / inverse-Euclidean and Euclidean / log-Euclidean). Thus, the symmetry and the positivity of the balanced bilinear forms were obvious. From now on, we build new bilinear forms thanks to the principle of balanced metrics. Therefore, it is not as obvious as before that these bilinear forms are metrics.

The family of power-Euclidean metrics $g^{E,\theta}$ \cite{Dryden10} indexed by the power $\theta\ne 0$ is defined by pullback of the Euclidean metric by the power function $\pow_\theta=\exp\circ\,\theta\log:(\M,\theta^2g^{E,\theta})\lto(\M,g^E)$:
\begin{equation}
    g^{E,\theta}_\Sigma(X,Y)=\frac{1}{\theta^2}\tr(\partial_X\pow_\theta(\Sigma)\,\partial_Y\pow_\theta(\Sigma)).
\end{equation}
This family comprise the Euclidean metric for $\theta=1$ and tends to the log-Euclidean metric when the power $\theta$ goes to 0. Therefore, we abusively consider that the log-Euclidean metric belongs to the family and we denote it $g^{E,0}:=g^{LE}$.

We define the mixed-power-Euclidean metrics $g^{E,\theta_1,\theta_2}$ as the balanced bilinear form of the power-Euclidean metrics $g^{E,\theta_1}$ and $g^{E,\theta_2}$, where $\theta_1,\theta_2\in\R$:
\begin{equation}
    g^{E,\theta_1,\theta_2}_\Sigma(X,Y)=\frac{1}{\theta_1\theta_2}\tr(\partial_X\pow_{\theta_1}(\Sigma)\,\partial_Y\pow_{\theta_2}(\Sigma)).
\end{equation}

Note that the family of mixed-power-Euclidean metrics contains the BKM metric for $(\theta_1,\theta_2)=(1,0)$ and the $\theta$-power-Euclidean metric for $(\theta_1,\theta_2)=(\theta,\theta)$, including the Euclidean metric for $\theta=1$ and the log-Euclidean metric for $\theta=0$.

At this stage, we do not know that the bilinear form $g^{E,\theta_1,\theta_2}$ is a metric. This is stated by Theorem \ref{MPE}. As the power-Euclidean metrics are flat, Theorem \ref{dual} combined with Theorem \ref{MPE} ensure that the Levi-Civita connections $\nabla^{E,\theta_1}$ and $\nabla^{E,\theta_2}$ of the metrics $g^{E,\theta_1}$ and $g^{E,\theta_2}$ are dual with respect to the $(\theta_1,\theta_2)$-mixed-power-Euclidean metric. This is stated by Corollary \ref{dualMPE}.

\begin{theorem}\label{MPE}
The bilinear form $g^{E,\theta_1,\theta_2}$ is symmetric and positive definite so it is a metric on ${SPD}_n$. Moreover, the symmetry ensures that $g^{E,\theta_1,\theta_2}=g^{E,\theta_2,\theta_1}$.
\end{theorem}

\begin{corollary}[$\theta_1$ and $\theta_2$-power-Euclidean are dual with respect to $(\theta_1,\theta_2)$-mixed-power-Euclidean]\label{dualMPE}
For $\theta_1,\theta_2\in\R$, we denote $\nabla^{E,\theta_1}$ and $\nabla^{E,\theta_2}$ the Levi-Civita connections of the power-Euclidean metrics $g^{E,\theta_1}$ and $g^{E,\theta_2}$. Then $g^{E,\theta_1,\theta_2}$ is the balanced metric of $g^{E,\theta_1},g^{E,\theta_2}$ and $({SPD}_n,g^{E,\theta_1,\theta_2},\nabla^{E,\theta_1},\nabla^{E,\theta_2})$ is a dually flat manifold.
\end{corollary}

To prove Theorem \ref{MPE}, we show that for all spectral decomposition $\Sigma=PDP^\top$ of an SPD matrix, there exists a matrix $A$ with positive coefficients $A(i,j)>0$ such that $g_\Sigma^{E,\theta_1,\theta_2}(X,Y)=\tr((A\bullet P^\top XP)(A\bullet P^\top YP))$, where $\bullet$ is the Hadamard product, i.e. $(A\bullet B)(i,j)=A(i,j)B(i,j)$, which is associative, commutative, distributive w.r.t. matrix addition and satisfies $\tr((A\bullet B)C)=\tr(B(A\bullet C))$ for symmetric matrices $A,B,C\in\Sym_n$. The existence of $A$ relies on Lemma \ref{diff}.

\begin{lemma}\label{diff}
Let $\Sigma=PDP^\top$ be a spectral decomposition of $\Sigma\in\M$, with $P\in O(n)$ and $D$ diagonal. For $f\in\{\exp,\log,\pow_\theta\}$, $\partial_Vf(\Sigma)=P(\delta(f,D)\bullet P^\top VP)P^\top$ where $\delta(f,D)(i,j)=\frac{f(d_i)-f(d_j)}{d_i-d_j}$. Note that $\frac{1}{\theta}\delta(\pow_\theta,D)(i,j)>0$ for all $\theta\in\R^*$ and $\delta(\log,D)(i,j)>0$.
\end{lemma}

\begin{proof}[Lemma \ref{diff}]
Once shown for $f=\exp$, it is easy to get for $f=\log$ by inversion and for $f=\pow_\theta=\exp\circ\,\theta\log$ by composition. But the case $f=\exp$ itself reduces to the case $f=\pow_k$ with $k\in\N$ by linearity, so we focus on this last case. As $\partial_V\pow_k(\Sigma)=\sum_{l=0}^{k-1}{\Sigma^lV\Sigma^{k-1-l}}=P\partial_{P^\top VP}\pow_k(D)P^\top$ and $\partial_{P^\top VP}\pow_k(D)(i,j)=\sum_{l=0}^{k-1}{D^lP^\top VPD^{k-1-l}(i,j)}=\frac{d_i^k-d_j^k}{d_i-d_j}P^\top VP(i,j)$, we get $\partial_V\pow_k(\Sigma)=P(\delta(\pow_k,D)\bullet P^\top VP)P^\top$.
\end{proof}

\begin{proof}[Theorem \ref{MPE}]
Let $\theta_1,\theta_2\in\R^*$. For a spectral decomposition $\Sigma=PDP^\top$, the matrix $A$ defined by $A(i,j)=\sqrt{\frac{1}{\theta_1}\delta(\pow_{\theta_1},D)(i,j)\frac{1}{\theta_2}\delta(\pow_{\theta_2},D)(i,j)}>0$ satisfies $g_\Sigma^{E,\theta_1,\theta_2}(X,Y)=\tr((A\bullet P^\top XP)(A\bullet P^\top YP))$. Symmetry and non-negativity are clear since they come from the Frobenius scalar product. Finally, if $g_\Sigma^{E,\theta_1,\theta_2}(X,X)=0$, then $A\bullet P^\top XP=0$ so $X=0$. So $g^{E,\theta_1,\theta_2}$ is a metric.
If $\theta_1=0$, the matrix $A$ defined by $A(i,j)=\sqrt{\delta(\log,D)(i,j)\frac{1}{\theta_2}\delta(\pow_{\theta_2},D)(i,j)}>0$ satisfies the same property and $g^{E,0,\theta_2}$ is a metric. $\Box$
\end{proof}

\section{The family of mixed-power-affine metrics}

In previous sections, we  defined our balanced metric from a pair of two flat metrics and we showed that it corresponded to the duality of (Levi-Civita) connections in information geometry. In this section, we investigate the balanced metric of two non-flat metrics and we observe that the corresponding Levi-Civita connections cannot be dual with respect to this balanced metric.

The family of power-affine metrics $g^{A,\theta}$ \cite{Thanwerdas19} indexed by the power $\theta\ne 0$ are defined by pullback of the affine-invariant metric by the power function $\pow_\theta:(\M,\theta^2g^{A,\theta})\lto(\M,g^A)$:
\begin{align}
    g^{A,\theta}_\Sigma(X,Y) &=\frac{1}{\theta^2}\tr(\Sigma^{-\theta}\,\partial_X\pow_\theta(\Sigma)\,\Sigma^{-\theta}\,\partial_Y\pow_\theta(\Sigma))
\end{align}
This family comprise the affine-invariant metric for $\theta=1$ and tends to the log-Euclidean metric when the power $\theta$ goes to 0. We consider that the log-Euclidean metric belongs to the family and we denote $g^{A,0}:=g^{LE}$.

As these metrics have no cut locus because they endow the manifold with a negatively curved Riemannian symmetric structure, there exists a unique geodesic between two given points. Therefore, a canonical parallel transport can be defined along geodesics. This allows to define the balanced bilinear form of two metrics without cut locus.

\begin{definition}[Balanced bilinear form]\label{balanced2}
Let $g,g^*$ be two metrics without cut locus on ${SPD}_n$ and $\Pi,\Pi^*$ the associated geodesic parallel transports. We define the balanced bilinear form $g^0_\Sigma(X,Y)=\tr((\Pi_{\Sigma\TO I_n}X)(\Pi^*_{\Sigma\TO I_n}Y))$.
\end{definition}

Given that the geodesic parallel transport on the manifold $(\M,g^{A,\theta})$ is $\Pi_{\Sigma\TO I_n}:X\in T_\Sigma\M\lmto\frac{1}{\theta}\Sigma^{-\theta/2}\partial_X\pow_\theta(\Sigma)\Sigma^{-\theta/2}\in T_{I_n}\M$, we define the mixed-power-affine metrics $g^{A,\theta_1,\theta_2}$ as the balanced metric of the power-affine metrics $g^{A,\theta_1}$ and $g^{A,\theta_2}$, where $\theta_1,\theta_2\in\R$ and $\theta=(\theta_1+\theta_2)/2$:
\begin{equation}
    g^{A,\theta_1,\theta_2}_\Sigma(X,Y)=\frac{1}{\theta_1\theta_2}\tr(\Sigma^{-\theta}\,\partial_X\pow_{\theta_1}(\Sigma)\,\Sigma^{-\theta}\,\partial_Y\pow_{\theta_2}(\Sigma)).
\end{equation}

Note that the family of mixed-power-affine metrics contains the $\theta$-power-affine metric for $(\theta_1,\theta_2)=(\theta,\theta)$, including the affine-invariant metric for $\theta=1$ and the log-Euclidean metric for $\theta=0$. This family has two symmetries since $g^{A,\theta_1,\theta_2}=g^{A,\pm\theta_1,\pm\theta_2}$, they come from the inverse-consistency of the affine-invariant metric. This family has a non-empty intersection with the family of mixed-power-Euclidean metrics since $g^{A,\theta_1,-\theta_1}=g^{E,\theta_1,-\theta_1}=g^{A,\theta_1}$ for all $\theta_1\in\R$.

The fact that $g^{A,\theta_1,\theta_2}$ is a metric can be shown exactly the same way as in the mixed-power-Euclidean case thanks to the equality $\Sigma^{-\theta/2}\partial_V\pow_\theta(\Sigma)\Sigma^{-\theta/2}=P(\varepsilon(\pow_\theta,D)\bullet P^\top VP)P^\top$ where $\varepsilon(\pow_\theta,D)=(d_id_j)^{-\theta/2}\delta(\pow_\theta,D)$ and where $\delta(\pow_\theta,D)$ has been defined in Lemma \ref{diff}. This is stated in Theorem \ref{MPA}.

\begin{theorem}\label{MPA}
The bilinear form $g^{A,\theta_1,\theta_2}$ is symmetric and positive definite. Hence it is a metric on ${SPD}_n$ and $g^{A,\theta_1,\theta_2}=g^{A,\theta_2,\theta_1}$.
\end{theorem}

Power-affine metrics being non-flat, $(\M,g^{A,\theta_1,\theta_2},\nabla^{A,\theta_1},\nabla^{A,\theta_2})$, where $\nabla^{A,\theta_1}$ and $\nabla^{A,\theta_2}$ are Levi-Civita connections of $g^{A,\theta_1}$ and $g^{A,\theta_2}$, cannot be a dually-flat manifold. Actually, the two connections are even not dual. It can be understood by comparison with previous sections since the duality was a consequence of the independence of the parallel transport with respect to the chosen curve, which was a consequence of the flatness of the two connections. Moreover, in the Definition \ref{balanced2}, the vectors are parallel transported along two different curves (the geodesics relative to each connection) so it may exists a better definition for the balanced bilinear form of two metrics without cut locus or even of two general metrics.

\section{Conclusion}
The principle of balanced bilinear form is a procedure on SPD matrices that takes a pair of flat metrics or metrics without cut locus and builds a new metric based on the parallel transport of the initial metrics. When the two initial metrics are flat, we showed that the two Levi-Civita connections are dual with respect to the balanced metric. When the two initial metrics are not flat, the two Levi-Civita connections seem not to be dual, so the principle of balanced metrics does not reduce to the concept of dual Levi-Civita connections. A challenging objective for future works is to define properly this principle for other general pairs of metrics and to find conditions under which the balanced bilinear form is a metric.

\paragraph{Acknowledgements.} This project has received funding from the European Research Council (ERC) under the European Union’s Horizon 2020 research and innovation programme (grant agreement No 786854). This work has been supported by the French government, through the UCAJEDI Investments in the Future project managed by the National Research Agency (ANR) with the reference number ANR-15-IDEX-01.

\bibliographystyle{unsrt}

\section{Appendix: proofs}

Let us first recall the statements of Theorem \ref{alpha} and Lemma \ref{expectations}.

~\\
\textbf{Theorem \ref{alpha}. }$\alpha$\textbf{-connections of the centered multivariate normal model}
\textit{In the global basis $(\partial_i)_{1\ls i\ls N=n(n+1)/2}$ of $\M={SPD}_n$ given by the inclusion $\M\hookrightarrow\Sym(n)\simeq\R^N$, writing $\partial_XY=X^i(\partial_i Y^j)\partial_j$, the $\alpha$-connections of the multivariate centered normal model are given by the following formula (\ref{alphaconnection}):
\begin{equation*}
    \nabla^{(\alpha)}_XY=\partial_XY-\frac{1+\alpha}{2}(X\Sigma^{-1}Y+Y\Sigma^{-1}X).
\end{equation*}
The mixture $m$-connection ($\alpha=-1$) is the Levi-Civita connection of the Euclidean metric $g^E_\Sigma(X,Y)=\tr(XY)$. The exponential $e$-connection ($\alpha=1$) is the Levi-Civita connection of the inverse-Euclidean metric, i.e. the pullback of the Euclidean metric by matrix inversion, $g^I_\Sigma(X,Y)=\tr(\Sigma^{-2}X\Sigma^{-2}Y)$.}\\

The formula (\ref{alphaconnection}) can be proved thanks to Lemma \ref{expectations} which gives the results of expressions of type $\int_{\R^n}{x^\top\Sigma^{-1}X\Sigma^{-1}Y\Sigma^{-1}Z\Sigma^{-1}x\exp\left(-\frac{1}{2}x^\top\Sigma^{-1}x\right)dx}$, with $y=\Sigma^{-1/2}x$, $A=\Sigma^{-1/2}X\Sigma^{-1/2}$, $B=\Sigma^{-1/2}Y\Sigma^{-1/2}$ and $C=\Sigma^{-1/2}Z\Sigma^{-1/2}$. If one wants to avoid using the third formula of Lemma \ref{expectations}, one can rely on the formula (\ref{alphaconnection}) in the case $\alpha=0$ which is already known from \cite{Skovgaard84}.

~\\
\textbf{Lemma \ref{expectations}.}
\textit{For $A,B,C\in\Sym_n$:
\begin{align*}
    \E_{I_n}(y\lmto y^\top Ay) &=\tr(A),\\
    \E_{I_n}(y\lmto y^\top Ayy^\top By) &=\tr(A)\tr(B)+2\tr(AB),\\
    \E_{I_n}(y\lmto y^\top Ayy^\top Byy^\top Cy) &=\tr(A)\tr(B)\tr(C)+8\tr(ABC)\\
    & \quad +2(\tr(AB)\tr(C)+\tr(BC)\tr(A)+\tr(CA)\tr(B)).
\end{align*}}

\subsection{Proof of Lemma \ref{expectations}}

Let $A,B,C\in\Sym_n$ and let $a_1\gs...\gs a_n$, $b_1\gs...\gs b_n$ and $c_1\gs...\gs c_n$ be their respective eigenvalues, $A=PA'P^\top$, $B=QB'Q^\top$, $C=RC'R^\top$ with $P,Q,R\in O(n)$ and $A'=\diag(a_1,...,a_n)$, $B'=\diag(b_1,...,b_n)$ and $C'=\diag(c_1,...,c_n)$.

We recall that if $X\sim\mc{N}(0,1)$, then $\E(X^{2k+1})=0$, $\E(X^2)=1$, $\E(X^4)=3$ and $\E(X^6)=15$.

\subsubsection{First equality}

The first equality is well known:
\begin{align*}
    \E_{I_n}(y\lmto y^\top Ay) &=\int_{\R^n}{y^\top Ay\frac{1}{\sqrt{2\pi}^n}\exp(y^\top y)dy}\\
    &=\int_{\R^n}{z^\top A'z\frac{1}{\sqrt{2\pi}^n}\exp(z^\top z)dz}\\
    &=\sum_{i=1}^n{\int_\R{a_iz_i^2\frac{1}{\sqrt{2\pi}^n}\exp(z_i^2)dz_i}}\\
    &=\sum_{i=1}^n{a_i}=\tr(A).
\end{align*}

\subsubsection{Second equality}

Due to the non commutativity a priori of the matrices $A$ and $B$, they are not diagonalizable in the same orthonormal basis so the calculus needs a bit more attention:
\begin{align*}
    \E_{I_n}(y\lmto y^\top Ayy^\top By)
    &=\int_{\R^n}{y^\top Ayy^\top By\frac{1}{\sqrt{2\pi}^n}\exp(y^\top y)dy}\\
    &=\int_{\R^n}{z^\top A'zz^\top U^\top B'Uz\frac{1}{\sqrt{2\pi}^n}\exp(z^\top z)dz}\\
    &=\int_{\R^n}{\left(\sum_{i=1}^n{a_iz_i^2}\right)\left(\sum_{j=1}^n{b_j(Uz)_j^2}\right)\frac{1}{\sqrt{2\pi}^n}\exp(z^\top z)dz},
\end{align*}
where $U=Q^\top P\in O(n)$. But $(Uz)_j^2=\sum_{k=1}^n{\sum_{l=1}^n{U(j,k)U(j,l)z_kz_l}}$ so:

\begin{align*}
    \E_{I_n}(y\lmto y^\top Ayy^\top By)
    &=\sum_{i,j,k,l}{a_ib_jU(j,k)U(j,l)\int_{\R^n}{{z_i^2z_kz_l}\prod_{m=1}^n{\frac{1}{\sqrt{2\pi}}\exp(z_m^2)}dz_m}}.
\end{align*}

According to the values of $i,k,l$, the integral takes different values:
\begin{enumerate}
    \item if $k=l=i$, it is equal to $\E(X^4)=3=2+1$,
    \item if $k=l\ne i$, it is equal to $\E(X^2)^2=1$,
    \item elsewhere, it is equal to $0$.
\end{enumerate}

\begin{align*}
    \E_{I_n}(y\lmto y^\top Ayy^\top By)
    &=\sum_{i,j}{a_ib_j[(2+1)U(j,i)U(j,i)+\sum_{k\ne i}{U(j,k)U(j,k)}]}\\
    &=\sum_{i,j}{a_ib_j[2U(j,i)^2+\sum_{k=1}^n{U(j,k)^2}]}\\
    &=2\tr(A'U^\top B'U)+\tr(A')\tr(B')\\
    &=2\tr(AB)+\tr(A)\tr(B).
\end{align*}

\subsubsection{Third equality}

The third equality is a bit more tedious to obtain because of the greater number of cases and especially because the rearrangement of the terms is more complex. Let $U=Q^\top P$ and $V=R^\top P\in O(n)$. Then:
\begin{align*}
    &\E_{I_n}(y\lmto y^\top Ayy^\top Byy^\top Cy)\\
    &=\int_{\R^n}{y^\top Ayy^\top Byy^\top Cy\frac{1}{\sqrt{2\pi}^n}\exp(y^\top y)dy}\\
    &=\int_{\R^n}{z^\top A'zz^\top U^\top B'Uzz^\top V^\top C'Vz\frac{1}{\sqrt{2\pi}^n}\exp(z^\top z)dz}\\
    &=\int_{\R^n}{\left(\sum_{i=1}^n{a_iz_i^2}\right)\left(\sum_{j=1}^n{b_j(Uz)_j^2}\right)\left(\sum_{j=1}^n{c_k(Vz)_k^2}\right)\frac{1}{\sqrt{2\pi}^n}\exp(z^\top z)dz}\\
    &=\sum_{i,j,k,l,m,r,s}{a_ib_jc_kU(j,l)U(j,m)V(k,r)V(k,s)\int_{\R^n}{{z_i^2z_lz_mz_rz_s}\prod_{t=1}^n{\frac{1}{\sqrt{2\pi}}\exp(z_t^2)}dz_t}}.
\end{align*}

According to the values of $i,l,m,r,s$, the integral takes different values:
\begin{enumerate}
    \item if $l=m=r=s=i$, it is equal to $\E(X^6)=15=6\times 3-3$,
    \item if $l=r=i\ne m=s$, it is equal to $\E(X^4)\E(X^2)=3$,
    \item if $l=s=i\ne m=r$, it is equal to $\E(X^4)\E(X^2)=3$,
    \item if $m=r=i\ne l=s$, it is equal to $\E(X^4)\E(X^2)=3$,
    \item if $m=s=i\ne l=r$, it is equal to $\E(X^4)\E(X^2)=3$,
    \item if $l=m=i\ne r=s$, it is equal to $\E(X^4)\E(X^2)=3$,
    \item if $l=m\ne i=r=s$, it is equal to $\E(X^4)\E(X^2)=3$,
    \item if $l=m=r=s\ne i$, it is equal to $\E(X^4)\E(X^2)=3=1+1+1$,
    \item if $l=m\ne r=s\ne i\ne l$, it is equal to $\E(X^2)^3=1$,
    \item if $l=r\ne m=s\ne i\ne l$, it is equal to $\E(X^2)^3=1$,
    \item if $l=s\ne m=r\ne i\ne l$, it is equal to $\E(X^2)^3=1$,
    \item elsewhere, it is equal to $0$.
\end{enumerate}

Instead of two terms as in the second equality, we have now eleven terms to rearrange. Lines 2 to 5 lack the term of line 1 to be identified as a trace term. Line 2 completed with line 1 gives the following term:
$$3\sum_{i,j,k}{a_ib_jc_k\sum_{m=1}^n{U(j,i)U(j,m)V(k,i)V(k,m)}}=3\tr(ABC)$$
and we can be convinced by symmetry of the indexes that lines 3, 4, 5 give the same result.

Lines 6 and 7 also lack the term of line 1 to be identified as a trace term. Line 6 completed with line 1 gives the following term:
$$3\sum_{i,j,k}{a_ib_jc_k\sum_{r=1}^n{U(j,i)U(j,i)V(k,r)V(k,r)}}=3\tr(AB)\tr(C)$$
and we can be convinced by symmetry of the indexes that line 7 gives $3\tr(AC)\tr(B)$.

Let us write the equality with the remaining terms, after distributing line 8 to lines 9, 10, 11 and merging lines 10 and 11:
\begin{align*}
    &\E_{I_n}(y\lmto y^\top Ayy^\top Byy^\top Cy)\\
    &=-3\sum_{i,j,k}{a_ib_jc_kU(j,i)^2V(k,i)^2} +12\tr(ABC)+3\tr(AB)\tr(C)+3\tr(AC)\tr(B)\\
    &+\sum_{i,j,k}{a_ib_jc_k\sum_{l\ne i}{\left[\sum_{r\ne i}{U(j,l)^2V(k,r)^2}+2\sum_{m\ne i}{U(j,l)U(j,m)V(k,l)V(k,m)}\right]}}
\end{align*}

We compute the last term of the previous expression in two steps. In the first step, we fix $i,j,k$ and we add and remove terms $r=i,m=i,l=i$. In the second step, we sum for $i,j,k\in\llbracket 1,n\rrbracket$.
\begin{align*}
    &\sum_{l\ne i}{\left[\sum_{r\ne i}{U(j,l)^2V(k,r)^2}+2\sum_{m\ne i}{U(j,l)U(j,m)V(k,l)V(k,m)}\right]}\\
    &=\sum_{l\ne i}{\left[\sum_{r}{U(j,l)^2V(k,r)^2}+2\sum_{m}{U(j,l)U(j,m)V(k,l)V(k,m)}\right]}\\
    &\quad -\sum_{l\ne i}{\left[U(j,l)^2V(k,i)^2+2U(j,l)U(j,i)V(k,l)V(k,i)\right]}\\
    &=\sum_{l}{\left[\sum_{r}{U(j,l)^2V(k,r)^2}+2\sum_{m}{U(j,l)U(j,m)V(k,l)V(k,m)}\right]}\\
    &\quad -\sum_{l}{\left[U(j,l)^2V(k,i)^2+2U(j,l)U(j,i)V(k,l)V(k,i)\right]}\\
    &\quad -\left[\sum_{r}{U(j,i)^2V(k,r)^2}+2\sum_{m}{U(j,i)U(j,m)V(k,i)V(k,m)}-3U(j,i)^2V(k,i)^2\right]\\
    &=1+2UV^\top(j,k)^2-V(k,i)^2-2UV^\top(j,k)U(j,i)V(k,i)\\
    &\quad -U(j,i)^2-2U(j,i)V(k,i)UV^\top(j,k)+3U(j,i)^2V(k,i)^2
\end{align*}

\begin{align*}
    &\sum_{i,j,k}{a_ib_jc_k\sum_{l\ne i}{\left[\sum_{r\ne i}{U(j,l)^2V(k,r)^2}+2\sum_{m\ne i}{U(j,l)U(j,m)V(k,l)V(k,m)}\right]}}\\
    &=\tr(A)\tr(B)\tr(C)+2\tr(A)\tr(BC)-(\tr(AB)\tr(C)+\tr(AC)\tr(B))-4\tr(ABC)\\
    &\quad +3\sum_{i,j,k}{a_ib_jc_kU(j,i)^2V(k,i)^2}
\end{align*}

Then, the remaining uncomputed term cancels with the first term in $\E_{I_n}(y\lmto y^\top Ayy^\top Byy^\top Cy)$. Finally, we get the third formula of Lemma \ref{expectations}:

\begin{align*}
    \E_{I_n}(y\lmto y^\top Ayy^\top Byy^\top Cy) &=\tr(A)\tr(B)\tr(C)+8\tr(ABC)\\
    & \quad +2(\tr(AB)\tr(C)+\tr(BC)\tr(A)+\tr(CA)\tr(B)).
\end{align*}

\subsection{Proof of Theorem \ref{alpha}}

\subsubsection{Proof of formula (\ref{alphaconnection})}

The Christoffel symbols are given by \cite{Amari00}:
\begin{equation*}
    \left(\Gamma_{ijk}^{(\alpha)}\right)_\Sigma=\E_\Sigma\left[\left(\partial_i\partial_j l+\frac{1-\alpha}{2}\partial_i l\partial_j l\right)\partial_k l\right].
\end{equation*}
where $l$ is the log-likelihood of the centered multivariate normal model, namely $l_\Sigma(x)=\log p_\Sigma(x)=\frac{1}{2}\left(-n\log(2\pi)-\log\det\Sigma+x^\top\Sigma^{-1}x\right)$.

So we have to compute $\E_\Sigma[\partial_i\partial_j l\partial_k l]$ and $\E_\Sigma[\partial_i l\partial_j l\partial_k l]$. We can either use the formulas of Lemma \ref{expectations} to compute each, or compute the first term and deduce the second one thanks to the formula of the Levi-Civita connection \cite{Skovgaard84} corresponding to the case $\alpha=0$.

We recall that $(Z^k\partial_kl)_{|\Sigma}(x)=d_\Sigma l(Z)(x)=-\frac{1}{2}[\tr(\Sigma^{-1}Z)+x^\top\Sigma^{-1}Z\Sigma^{-1}x]$ (cf. table in section 2.3). As a consequence, deriving once more in this basis, we get: $(X^iY^j\partial_i\partial_jl)_{|\Sigma}(x)=\frac{1}{2}[\tr(\Sigma^{-1}X\Sigma^{-1}Y)+x^\top\Sigma^{-1}(X\Sigma^{-1}Y+Y\Sigma^{-1}X)\Sigma^{-1}x]$. Denoting $A=\Sigma^{-1/2}(X\Sigma^{-1}Y+Y\Sigma^{-1}X)\Sigma^{-1/2}$ and $B=\Sigma^{-1/2}Z\Sigma^{-1/2}$, we can now compute the first term in the Christoffel symbols:
$X^iY^jZ^k\E_\Sigma[\partial_i\partial_j l\partial_k l]
=\frac{1}{4}(\tr(\Sigma^{-1}X\Sigma^{-1}Y)\tr(\Sigma^{-1}Z)(-1+1)+\tr(A)\tr(B)-(2\tr(AB)+\tr(A)\tr(B)))=-\frac{1}{2}\tr(AB)=-\tr(\Sigma^{-1}X\Sigma^{-1}Y\Sigma^{-1}Z)$.\\

Let us compute the second term in the Christoffel symbols thanks to the two previously detailed methods.

\paragraph{1\textsuperscript{st} method: using the third equality of Lemma \ref{expectations}}~

Thanks to the expression of $\partial_il$ given above, and denoting $A=\Sigma^{-1/2}X\Sigma^{-1/2}$, $B=\Sigma^{-1/2}Y\Sigma^{-1/2}$ and $C=\Sigma^{-1/2}Z\Sigma^{-1/2}$, we are now able to compute
$X^iY^jZ^k\E_\Sigma[\partial_i l\partial_j l\partial_k l]
=\frac{1}{8}(-\tr(A))\tr(B)\tr(C)(-1+1+1+1)-\tr(A)(2\tr(BC)+\tr(B)\tr(C))-\tr(B)(2\tr(AC)+\tr(A)\tr(C))-\tr(C)(2\tr(AB)+\tr(A)\tr(B))+\frac{1}{8}(8\tr(ABC)+2(\tr(AB)\tr(C)+\tr(AC)\tr(B)+\tr(BC)\tr(A))+\tr(A)\tr(B)\tr(C)))=\tr(ABC)=\tr(\Sigma^{-1}X\Sigma^{-1}Y\Sigma^{-1}Z)$.

\paragraph{2\textsuperscript{nd} method: using the formula of the Levi-Civita connection}~

From \cite{Skovgaard84}, we know that $X^iY^jZ^k\left(\Gamma_{ijk}^{(0)}\right)_\Sigma=-\frac{1}{2}\tr(\Sigma^{-1}X\Sigma^{-1}Y\Sigma^{-1}Z)$. On the other hand, $X^iY^jZ^k\left(\Gamma_{ijk}^{(0)}\right)_\Sigma=X^iY^jZ^k\E_\Sigma(\partial_i\partial_jl\partial_kl)+\frac{1}{2}X^iY^jZ^k\E(\partial_il\partial_jl\partial_kl)$ and we already computed $X^iY^jZ^k\E_\Sigma(\partial_i\partial_jl\partial_kl)=-\tr(\Sigma^{-1}X\Sigma^{-1}Y\Sigma^{-1}Z)$. We can deduce the second term $X^iY^jZ^k\E(\partial_il\partial_jl\partial_kl)=\tr(\Sigma^{-1}X\Sigma^{-1}Y\Sigma^{-1}Z)$.\\

To conclude:
\begin{align*}
    g_\Sigma(\nabla^{(\alpha)}_XY-\partial_XY,Z)&=X^iY^jZ^k\left(\Gamma_{ijk}^{(\alpha)}\right)_\Sigma\\
    &=-\frac{1+\alpha}{2}\tr(\Sigma^{-1}X\Sigma^{-1}Y\Sigma^{-1}Z)\\
    &=-\frac{1+\alpha}{4}\tr(\Sigma^{-1}(X\Sigma^{-1}Y+Y\Sigma^{-1}X)\Sigma^{-1}Z)
\end{align*} so $(\nabla^{(\alpha)}_XY)_{|\Sigma}=(\partial_XY)_{|\Sigma}-\frac{1+\alpha}{2}(X\Sigma^{-1}Y+Y\Sigma^{-1}X)$ and formula (\ref{alphaconnection}) is proved.

\subsubsection{Mixture and exponential connections}

For $\alpha=-1$, there only remains the Levi-Civita connection of the Euclidean metric $\partial$ in the formula: this is the mixture $m$-connection. On the other hand, the exponential $e$-connection ($\alpha=1$) is given by $\nabla^{(e)}_XY=\partial_XY-(X\Sigma^{-1}Y+Y\Sigma^{-1}X)$.

Let us compute the Levi-Civita connection $\nabla^I$ of the inverse-Euclidean metric given by $g^I_\Sigma(X,Y)=\tr(\Sigma^{-2}X\Sigma^{-2}Y)$ thanks to the Koszul formula:
\begin{align*}
    2g_\Sigma^I(\nabla^I_XY,Z)&=X_\Sigma(g^I(Y,Z))+Y_\Sigma(g^I(X,Z))-Z_\Sigma(g^I(X,Y))\\
    &\quad +g^I([X,Y],Z)-g^I([X,Z],Y)-g^I([Y,Z],X)\\
    &=-2\tr((\Sigma^{-2}X\Sigma^{-1}+\Sigma^{-1}X\Sigma^{-2})Y\Sigma^{-2}Z)\\
    &\quad -2\tr((\Sigma^{-2}Y\Sigma^{-1}+\Sigma^{-1}Y\Sigma^{-2})X\Sigma^{-2}Z)\\
    &\quad +2\tr((\Sigma^{-2}Z\Sigma^{-1}+\Sigma^{-1}Z\Sigma^{-2})X\Sigma^{-2}Y)\\
    &\quad +2\tr(\Sigma^{-2}\partial_XY\Sigma^{-2}Z)\\
    &=2\tr(\Sigma^{-2}[\partial_XY-(X\Sigma^{-1}Y+Y\Sigma^{-1}X)]\Sigma^{-2}Z)
\end{align*}
So $\nabla^I_XY=\partial_XY-(X\Sigma^{-1}Y+Y\Sigma^{-1}X)$ and $\nabla^{(e)}=\nabla^I$.

\end{document}